\documentclass[12pt,a4paper]{amsart}

\usepackage[utf8]{inputenc}
\usepackage[T1]{fontenc}
\usepackage[english]{babel}
\usepackage{cite}

\usepackage{indentfirst}
\usepackage{amssymb}
\usepackage{amsfonts}
\usepackage{amsmath}
\usepackage{amsthm}
\usepackage[top=2.5cm, bottom=2.5cm, left=2.5cm, right=2.5cm]{geometry}
\usepackage{amsopn}
\usepackage[colorlinks]{hyperref}
\usepackage{mathrsfs}
\usepackage{amssymb}
\usepackage{amsxtra}
\usepackage{color}
\usepackage{dsfont}
\usepackage{bbm}

\newcommand{\al}{\alpha}
\newcommand{\be}{\beta}

\newcommand{\R}{\mathbb{R}}

\newcommand{\Ss}{\mathbb{S}}

\newcommand{\Rd}{\mathbb{R}^d}
\newcommand{\eps}{\varepsilon}
\newcommand{\EEE}{\mathcal{E}}
\newcommand{\LL}{\mathcal{L}} 
\newcommand{\Cc}{\mathcal{C}} 
\newcommand{\BB}{\mathrm{B}} 

\DeclareMathOperator{\supp}{supp}
\DeclareMathOperator{\dist}{dist}

\newcommand{\LHS}{\mathrm{LHS}}
\newcommand{\RHS}{\mathrm{RHS}}

\newtheorem{theorem}{Theorem}

\newtheorem{lemma}{Lemma}

\theoremstyle{definition}

\newtheorem{remark}{Remark}

\title[Hardy inequalities]{Sharp Hardy inequalities for Sobolev-Bregman forms}
\date{\today}

\author[M.~Kijaczko]{Micha{\l} Kijaczko}
\address{Faculty of Pure and Applied Mathematics, Wroc\l{}aw University of Science and Technology, Wyb. Wyspia\'nskiego 27, 50-370 Wroc\l{}aw, Poland.}
\email{michal.kijaczko@pwr.edu.pl}

\author[J.~Lenczewska]{Julia Lenczewska}
\thanks{The second named author was partially supported by National Science Centre (Poland) grant 2019/33/B/ST1/02494.}
\address{Faculty of Pure and Applied Mathematics, Wroc\l{}aw University of Science and Technology, Wyb. Wyspia\'nskiego 27, 50-370 Wroc\l{}aw, Poland.}
\email{julia.lenczewska@pwr.edu.pl}


\begin{document}
\selectlanguage{english}

\maketitle
\begin{abstract}
We obtain sharp fractional Hardy inequalities for the half-space and for convex domains. We extend the results of Bogdan and Dyda and of Loss and Sloane to the setting of Sobolev-Bregman forms.
\end{abstract}

\section{Introduction and main results}

Let $0 < \al < 2$ and $d = 1, 2, \ldots$. Bogdan and Dyda \cite{MR2663757} proved the following Hardy inequality in the half-space $D = \{x=(x_1, \ldots, x_d) \in \Rd: x_d>0 \}$. 
For every $u\in C_c(D)$,
\begin{equation}\label{HardyBD}
\frac{1}{2}
\int_D \! \int_D
\frac{(u(x)-u(y))^2}{|x-y|^{d+\al}} \,dx\,dy
\geq {\kappa_{d,\al}}
\int_D \frac{u(x)^2}{x_d^{\al}} dx\,,
\end{equation}
where
\begin{equation}
  \label{eq:ns}
\kappa_{d,\al}=\frac{\pi^{\frac{d-1}{2}}
  \Gamma\big(\frac{1+\al}{2}\big)}{\Gamma\big(\frac{\al+d}{2}\big)}
\frac{\BB\big(\frac{1+\al}{2}, \frac{2-\al}{2}\big)-2^{\al}}{\al2^{\al}}\,,
\end{equation}
and {\rm \eqref{HardyBD}} fails to hold for some $u\in C_c(D)$ if $\kappa_{d,\al}$ is replaced by a
bigger constant. Here,
$\Gamma$ is the Euler gamma function, 
$B$ is the Euler beta function, and 
$C_c(D)$ denotes the class of all the continuous functions
$u\,:\; \Rd\to \R$ with compact support in $D$.

The main purpose of this note is to prove a generalization of
this inequality, 
where the left-hand side of \eqref{HardyBD} is replaced with the following form:
\begin{equation}\label{e:dEp}
\EEE_p[u] :=
\frac{1}{2}
\int_{D} \int_{D} (u(x)-u(y)) (u(x)^{\langle p - 1 \rangle}  -u(y)^{\langle p - 1 \rangle} ) |x-y|^{-d-\al} \, dy \, dx,
\end{equation}
defined for $p \in (1,\infty)$ and $u : \Rd \to \R$, where
$$
a^{\langle k \rangle} := \left|a \right|^k \operatorname{sgn} a,\quad a, k \in \R.
$$
We call such integral forms the \emph{Sobolev-Bregman forms}.

For $\al\neq 1$ let 
\begin{equation}
  \label{eq:ns1}
\kappa_{d,p, \al}= - \frac{\pi^{\frac{d-1}{2}}
  \Gamma\big(\frac{1+\al}{2}\big)}{\Gamma\big(\frac{\al+d}{2}\big)} 
	\left(\BB\left(\tfrac{\al-1}{p} +1, -\al\right) + \BB\left(\al-\tfrac{\al-1}{p}, -\al\right) + \tfrac{1}{\al}\right) \geq 0.
\end{equation}
Recall that $\BB(x,y)=\Gamma(x)\Gamma(y)/\Gamma(x+y)$ and $1/\Gamma$ can be extended analytically to the whole of $\R$, hence $\BB(x,y)$ is well defined for all $x,y\neq 0,-1,-2,\dots$. Noteworthy, $\kappa_{d,p,1}=0$ (understood as the limit of $\kappa_{d,p,\al}$ as $\al \to 1$). Furthermore, observe that  $\kappa_{d,p,\alpha} = \kappa_{d,p',\alpha}$, where $p'=p/(p-1)$.

Our first main result reads as follows.

\begin{theorem}\label{thm1}
Let $0<\al<2$, $d=1,2,\dots$ and $1<p<\infty$. For every $u\in C_c(D)$,
\begin{equation}\label{Hardyin}
\frac{1}{2} \int_D \int_D \frac{(u(x)-u(y))(u(x)^{\langle p-1 \rangle} -u(y)^{\langle p-1 \rangle})}{|x-y|^{d+\al}} \,dx\,dy
\geq
 {\kappa_{d,p,\al}}
\int_D \frac{|u(x)|^p}{x_d^{\al}} dx\,,
\end{equation}
and the constant in \eqref{Hardyin} is the best possible, i.e. it cannot be replaced by a bigger one.
\end{theorem}

Our work is motivated by a recent paper of Bogdan, Jakubowski, Lenczewska and Pietruska-Pa{\l}uba \cite{bogdan2021optimal},  
who obtained a similar inequality for the whole space $\R^d$ instead of $D$ (see \cite[Theorems 1 and 2]{bogdan2021optimal}), namely they proved that if $0<\alpha<d \wedge 2$, then for $u\in L^p(\mathbb R^d)$,
$$
\frac{\mathcal{A}_{d,-\al} }{2} \int_{\Rd} \int_{\Rd} \frac{(u(x)-u(y))(u(x)^{\langle p-1 \rangle} -u(y)^{\langle p-1 \rangle})}{|x-y|^{d+\al}} \,dx\,dy
\geq
 {\kappa_{\frac{d-\al}{p}}}
\int_{\Rd} \frac{|u(x)|^p}
{|x|^{\al}} dx \, ,
$$
where $\mathcal{A}_{d,-\al}=\frac{2^{\al}\Gamma((d+\al)/2)}{\pi^{d/2}|\Gamma(-\al/2)|}$ and the constant $\kappa_{\frac{d-\al}{p}}$ is explicit and optimal.
They used this inequality to characterize the $L^{p}$ contractivity property of the Feynman-Kac semigroup generated by $\Delta^{\al/2}+\kappa|x|^{-\al}$.

For the sake of completeness, let us mention that the
sharp fractional Hardy inequality
\begin{equation}\label{frankseringer}
\int_{D}\int_{D}\frac{|u(x)-u(y)|^{p}}{|x-y|^{d+sp}}\,dy\,dx\geq \mathcal{D}_{d,s,p}\int_{D}\frac{|u(x)|^{p}}{x_{d}^{sp}}\,dx,  
\end{equation}
where $u \in C_0^{\infty}(\overline{D})$ if $ps<1$ and $u\in C_0^{\infty} (D)$ if $ps>1$,
was obtained by Frank and Seiringer in \cite{MR2723817}. The constant $\mathcal{D}_{d,s,p}$ is optimal and has an explicit form (see \cite[(1.4)]{MR2723817}). However, this result is not directly comparable to ours, as the integral forms in \eqref{e:dEp} and on the left-hand side of \eqref{frankseringer} are different for $p\neq 2$. The reader interested in fractional Hardy inequalities may also see Frank and Seiringer \cite{MR2469027} for an analogous result on $\R^d$, Frank, Lieb and Seiringer \cite{MR2425175} for other optimal inequalities and \cite{MR2085428, MR3237044, MR3803664}
for more general Hardy inequalities, but with unknown sharp constants.

Loss and Sloane \cite{MR2659764} proved that 
if $\alpha \in (1,2)$, then
a fractional Hardy inequality similar to \eqref{frankseringer} holds for all convex, proper subsets of $\Rd$, with the same optimal constant (see \cite[Theorem 1.2]{MR2659764}). We obtain an analogous formula for Sobolev-Bregman forms and this is our second main result.

\begin{theorem}\label{convex}
Let $\Omega$ be an open, proper subset of $\Rd$ and let $1<\al<2$. Then, for $u\in C_{c}(\Omega)$,
\begin{equation}\label{hardyconvex1}
\frac{1}{2}\int_{\Omega} \int_{\Omega} 
\frac{(u(x)-u(y))(u(x)^{\langle p-1 \rangle} -u(y)^{\langle p-1 \rangle})}{|x-y|^{d+\al}} \,dx\,dy\geq\kappa_{d,p,\al}\int_{\Omega}\frac{|u(x)|^{p}}{m_{\al}(x)^{\al}}\,dx,
\end{equation}
where 
$$
m_{\al}(x)^{\al}=\frac{\int_{\Ss^{d-1}}|\omega_{d}|^{\al}\,d\omega}{\int_{\Ss^{d-1}}d_{\omega,\Omega}(x)^{-\al}\,d\omega}, \quad d_{\omega,\Omega}(x)=\min\{|t|:x+t\omega\notin\Omega\}.
\vspace{0.1cm}
$$
In particular, if $\Omega$ is convex, then
\begin{equation}\label{hardyconvex2}
 \frac{1}{2}\int_{\Omega} \int_{\Omega} \frac{(u(x)-u(y))(u(x)^{\langle p-1 \rangle} -u(y)^{\langle p-1 \rangle})}{|x-y|^{d+\al}} \,dx\,dy
 \geq 
 \kappa_{d,p,\al} 
 \int_{\Omega}\frac{|u(x)|^{p}}{\dist(x,\partial\Omega)^{\al}}\,dx. 
\end{equation}
The constant in \eqref{hardyconvex2} is optimal.
\end{theorem}
 Here, $\dist(x, \partial \Omega)$ denotes the distance from the point $x \in \Omega$ to $\partial \Omega$, i.e. $\dist(x,\partial\Omega)=\inf_{y\in\partial\Omega}|x-y|$. 
 We note that \eqref{hardyconvex2} is an easy consequence of \eqref{hardyconvex1} since $m_{\al}(x) \leq \dist(x, \partial \Omega)$ if $\Omega$ is convex and $x \in \Omega$ (see \cite{MR2659764}). If $\al\leq 1$ and $\Omega$ is a bounded convex domain, then the inequality \eqref{hardyconvex2} cannot hold with any positive constant, see Remark \ref{remark}.

Fractional Hardy inequalities are of interest not only from the analytical point of view, but also due to their connection with stochastic processes by means of
Dirichlet forms.
In particular, our result is related to the censored $\al$-stable process in $D$, which, 
informally speaking, 
is a stable process ,,forced'' to stay inside $D$, see Bogdan, Burdzy and Chen \cite{MR2006232}. If we denote its Dirichlet form by $\mathcal{C}(u,v)$, then similarly as in \cite{MR2663757},
\eqref{Hardyin} is for $u\in C_c^{\infty}(D)$ equivalent to 
\begin{equation*}
\Cc(u,u^{\langle p-1 \rangle})
\geq 
 \mathcal{A}_{d,-\alpha} \kappa_{d,p,\alpha}
\int_D \frac{|u(x)|^p}{x_d^{\al}} dx .
\end{equation*}
Due to the relation between $\mathcal{C}$ and the Dirichlet form of the $\al$-stable process killed upon leaving $D$, which we denote by $\mathcal{K}(u,v)$ (we again refer to \cite{MR2006232}), we get
\begin{equation*}
\mathcal{K}(u,u^{\langle p-1 \rangle})
\geq \mathcal{A}_{d,-\alpha} \left(\kappa_{d,p,\alpha} + \frac{1}{\alpha} \frac{\pi^{(d-1)/2} \Gamma\left((1+\alpha)/2\right)}{\Gamma\left((\alpha+d)/2\right)}\right)
\int_D \frac{|u(x)|^p}{x_d^{\alpha}} dx ,
\end{equation*}
for all $u\in C^\infty_c(D)$.

The generator of the censored $\alpha$-stable process in $D$ is the \emph{regional fractional Laplacian}, defined for $u \in C_c^2(D)$ by the formula
$$
\Delta^{\al/2}_D u(x)= \mathcal{A}_{d,-\al} \lim_{\varepsilon\to 0^+}
\int_{D\cap \{|y-x|>\varepsilon\}}
\frac{u (y)-u(x)}{|x-y|^{d+\al}} \,dy\,, 
$$  
see \cite{MR2238879, MR2214908}.
We will use the notation $\LL = \mathcal{A}_{d,-\al}^{-1} \Delta^{\al/2}_D.$

For $a,b\in\R$ and $p>1$, we define the \emph{Bregman divergence}
$$
F_{p}(a,b):=|b|^{p}-|a|^{p}-p a^{\langle p-1 \rangle}(b-a).
$$
The function $F_p$ is nonnegative as the second-order Taylor remainder of the convex function $x\mapsto |x|^{p}$. Moreover, we have the identity $F_{p}(a,b)+F_{p}(b,a)=p(b-a)(b^{\langle p-1 \rangle}-a^{\langle p-1 \rangle})$, and the latter expression appears on the left-hand side of \eqref{Hardyin}. We refer the reader to \cite{bogdan2021optimal} for more references on Sobolev-Bregman forms.

We denote by $|x|=(x_1^2+\dots+x_d^2)^{1/2}$  the Euclidean norm of 
$x=(x_1,\ldots,x_d)\in \R^d$, and $B(x,r)$ 
stands for the Euclidean ball of radius $r>0$ centered at $x$.
For $d\geq 2$ we write $x=(x',x_d)$, where
$x'=(x_1,\ldots,x_{d-1})$, and we let $\|x'\|=\displaystyle\max_{k=1,\ldots,d-1}
|x_k|$.

\textbf{Acknowledgement.} We thank Bart\l{}omiej Dyda and Tomasz Jakubowski for comments on the original version of the manuscript and inspiring discussions on related Hardy inequalities. We would also like to thank the anonymous referee for helpful remarks.

\section{Proof of Theorem \ref{thm1}}
In order to prove our first 
result, we will need
an analogue of \cite[Theorem 1]{bogdan2021optimal}. 

Let $w= w_{\be}=x_d^{\be}$ for $\be\in(-1,\al)$. By \cite[(5.4) and (5.5)]{MR2006232},
\begin{equation}\label{ullapl}
 \LL w_{\be}(x) = 
\;\gamma(\al,\be)\,   \frac{\pi^\frac{d-1}{2}  \Gamma\big(\tfrac{1+\al}{2}\big)}
  {\Gamma\big(\tfrac{\al+d}{2}\big)}\, x_d^{-\al}\,w_{\be}(x) \,,
\end{equation}
where 
\begin{equation}
  \label{eq:dg}
 \gamma(a,b) = \int_0^1 
  \frac{(t^{b}-1)(1-t^{a-b-1})}{(1-t)^{1+a}}\, dt,  \quad a\in(0,2),\,b\in(-1,a),
\end{equation}
\noindent is absolutely convergent. We note that $\gamma(\al, \be) \leq 0$ if and only if  $\be(\al-\be-1)\geq 0$.

\begin{lemma}\label{ep}
Let $w(x)=w_{\be}(x)=x_{d}^{\be},$
$\be \in (-1, \al)$ for $p \in (1,2)$ and $\be \in \big(-\frac{1}{p-1}, \frac{\al}{p-1}\big)$ for $p \in [2,\infty)$,
and $u\in C_{c}(D)$. Then we have
\begin{align*}
\EEE_{p}[u]&=
\frac{(p-1)\gamma(\al, \be) + \gamma(\al, (p-1)\be)}{p} \frac{\pi^\frac{d-1}{2}  \Gamma\big(\frac{1+\al}{2}\big)}
  {\Gamma\big(\frac{\al+d}{2}\big)}\,  \int_D \frac{|u(x)|^p}{x_d^{\al}} dx 
	\\&\quad 
+\frac{1}{p} \int_{D}\int_{D}F_{p}\left(\frac{u(x)}{w(x)},\frac{u(y)}{w(y)}\right)\,w(x)^{p-1}\,w(y)\frac{dy\,dx}{|x-y|^{d+\al}}.
\end{align*}
In particular, for $\be=\frac{\al-1}{p}$,
\begin{align*}
    \EEE_p[u] &= \kappa_{d,p,\al} \int_D \frac{|u(x)|^p}{x_d^{\al}} dx + \frac{1}{p} \int_{D}\int_{D}F_{p}\left(\frac{u(x)}{w(x)},\frac{u(y)}{w(y)}\right)\,w(x)^{p-1}\,w(y)\frac{dy\,dx}{|x-y|^{d+\al}}.
\end{align*}
\end{lemma}
\begin{proof}
Let $w= w_{\be}$, $u \in C_c(D)$, $x,y \in D$. We have
\begin{align}
&p u(x)^{\langle p-1 \rangle} (u(x)-u(y)) + |u(y)|^p \frac{w(x)^{p-1} - w(y)^{p-1}}{w(y)^{p-1}} + (p-1) |u(x)|^p \frac{w(y)-w(x)}{w(x)}\label{eq:dec} \\
&= |u(y)|^p \frac{w(x)^{p-1} - w(y)^{p-1}}{w(y)^{p-1}} + p |u(x)|^p \frac{w(y)-w(x)}{w(x)} - |u(x)|^p \frac{w(y)-w(x)}{w(x)}
\nonumber \\
&\quad 
- pu(x)^{\langle p-1 \rangle} (u(y)-u(x)) \nonumber\\
&=|u(y)|^p \frac{w(x)^{p-1} - w(y)^{p-1}}{w(y)^{p-1}} 
- |u(x)|^p \frac{w(y)-w(x)}{w(x)}
\nonumber\\
&\quad
- p \left(u(x)^{\langle p-1 \rangle}u(y) - |u(x)|^p - |u(x)|^p \frac{w(y)-w(x)}{w(x)} \right) \nonumber.
\end{align}

\noindent We integrate both sides
with respect to the measure 
$\mu_{\eps}(dx\,dy) :=\mathds{1}_{\{|x-y|>\eps\}} |x-y|^{-d-\al}\,dx\,dy$
and use
the symmetry of $\mu_{\eps}$. 
We obtain

\begin{align*}
\LHS &= \frac{p}{2} \int_{D}\int_{D} (u(x)-u(y)) (u(x)^{\langle p-1 \rangle} -u(y)^{\langle p-1 \rangle}) \, \mu_{\eps}(dx\,dy) \\
&\quad + \int_{D}\int_{D}|u(x)|^p \frac{w(y)^{p-1} - w(x)^{p-1}}{w(x)^{p-1}} \, \mu_{\eps}(dx\,dy) 
\\
&\quad+  (p-1)\int_{D}\int_{D}|u(x)|^p \frac{w(y)-w(x)}{w(x)} \, \mu_{\eps}(dx\,dy)
\end{align*}
and
\begin{align*}
 \RHS &=  \int_{D}\int_{D}\left(|u(y)|^p \frac{w(x)^{p-1}}{w(y)^{p-1}} - |u(y)|^p  -|u(x)|^p \frac{w(y)}{w(x)} + |u(x)|^p \right) \mu_{\eps}(dx\,dy)\\
&\quad- p \int_{D}\int_{D}\left(u(x)^{\langle p-1 \rangle}u(y) - |u(x)|^p - |u(x)|^p \frac{w(y)-w(x)}{w(x)} \right) \mu_{\eps}(dx\,dy) \\
&= \int_{D}\int_{D}\left(|u(y)|^p \frac{w(x)^{p-1}}{w(y)^{p-1}} -|u(x)|^p \frac{w(y)}{w(x)} -pu(x)^{\langle p-1 \rangle} u(y) + p|u(x)|^p \frac{w(y)}{w(x)} \right)\mu_{\eps}(dx\,dy) \\
&= \int_{D}\int_{D}\left(\frac{|u(y)|^p}{w(y)^p} -\frac{|u(x)|^p}{w(x)^p} -p \frac{u(x)^{\langle p-1 \rangle}}{w(x)^{p-1}} \left(\frac{u(y)}{w(y)} - \frac{u(x)}{w(x)} \right)\right) w(x)^{p-1} w(y) \, \mu_{\eps}(dx\,dy) \\
&= \int_{D}\int_{D} F_p\left(\frac{u(y)}{w(y)}, \frac{u(x)}{w(x)}\right) w(x)^{p-1} w(y) \, \mu_{\eps}(dx\,dy).
\end{align*}
We note that for $b \in \{1, p-1\}$,
$$
\int_{\{y \in D: |x-y|>\eps\}} \frac{w(x)^{b}-w(y)^{b}}{|x-y|^{d+\al}} dy \to 
-\LL w(x)^{b} \qquad \textnormal{as} \quad \eps \to 0,
$$
uniformly in $x \in \supp u$, thus letting $\eps \to 0$ yields

\begin{align*}
& \frac{p}{2} \int_D \int_D  (u(x)-u(y)) (u(x)^{\langle p-1 \rangle} -u(y)^{\langle p-1 \rangle}) |x-y|^{-d-\al} \, dx \, dy \\
&\quad =\int_D |u(x)|^p \lim_{\eps \to 0} \int_{\{y \in D:|x-y|> \eps\}}  \left(w(x)^{p-1} - w(y)^{p-1}\right) |x-y|^{-d-\al} \, dy  \, \frac{d x }{w(x)^{p-1}} \\
&\quad\quad +  (p-1) \int_D |u(x)|^p \lim_{\eps \to 0}\int_{\{y \in D: |x-y|>\eps\}} \left(w(x)-w(y)\right) |x-y|^{-d-\al} \, dy \, \frac{d x }{w(x)} \\
&\quad\quad + \int_{D}\int_{D} F_p\left(\frac{u(y)}{w(y)}, \frac{u(x)}{w(x)}\right) w(x)^{p-1} w(y) |x-y|^{-d-\al} \, dx \, dy \\
&\quad = \int_D |u(x)|^p  \left(\frac{-\LL w(x)^{p-1}}{w(x)^{p-1}} + (p-1) \frac{-\LL w(x)}{w(x)} \right) d x \\
&\quad\quad + \int_{D}\int_{D} F_p\left(\frac{u(y)}{w(y)}, \frac{u(x)}{w(x)}\right) w(x)^{p-1} w(y) |x-y|^{-d-\al} \,dx \,dy.
\end{align*}
Hence, the first assertion of the Lemma follows. Further,
\begin{align*}
&(p-1) \frac{\LL w (x)}{w(x)} + \frac{\LL w(x)^{p-1}}{w(x)^{p-1}} = \left( (p-1)\gamma(\al, \be) + \gamma(\al, (p-1)\be)\right) \frac{\pi^\frac{d-1}{2}  \Gamma\big(\frac{1+\al}{2}\big)}
  {\Gamma\big(\frac{\al+d}{2}\big)}\, x_d^{-\al}\\
	&\quad= \frac{\pi^\frac{d-1}{2}  \Gamma\big(\frac{1+\al}{2}\big)}
  {\Gamma\big(\frac{\al+d}{2}\big)} \int_0^1  \frac{(p-1)(t^{\be}-1)(1-t^{\al-\be-1})+ (t^{(p-1)\be}-1)(1-t^{\al-(p-1)\be-1})}{(1-t)^{1+\al}}\, d t \, x_d^{-\al}.
\end{align*} 
It is easy to check that, for $t \in(0,1)$, the minimum of the function
$$
\beta \mapsto (p-1)(t^{\be}-1)(1-t^{\al-\be-1}) + (t^{(p-1)\be}-1)(1-t^{\al-(p-1)\be-1})
$$
is $p(t^{\be}-1)(t^{\be'}-1)<0$, where $\be=\tfrac{\al-1}{p}$ and $\be'=\tfrac{(p-1)(\al-1)}{p}$. Further, 
since $\Gamma(x+1)=x\Gamma(x)$, by \cite[(2.2)]{MR2663757} we get
\begin{align}
\gamma(\al,\be)&=
\BB(\be+1,-\al)+\BB(\al-\be,-\al) + \frac{1}{\al} \label{eq:gamm},
\end{align}
for $\al\in(0,2)\setminus \{1\}$ and $\be\in(-1,\al)$.
Hence, the second assertion of the Lemma follows.
\end{proof}

\begin{proof}[Proof of Theorem~\ref{thm1}]
By Lemma \ref{ep}, \eqref{Hardyin} holds. 
To complete the proof, we will verify the
optimality of $\kappa_{d,p,\al}$. Our proof is a modification of \cite[Proof of Theorem 1.1]{MR2663757}.
In what follows let ${\be}=\frac{\al-1}{p}$.
If $\al\geq 1$, there are real functions $v_n$, $n=1,2, \ldots$, such that
\begin{itemize}
\item[(i)] $v_n=1$ on $[-n^2,n^2]^{d-1}\times \big[\frac{1}{n}, 1\big]$,
\item[(ii)] $\supp v_n \subset [-n^2-1,n^2+1]^{d-1}\times
  \big[\frac{1}{2n}, 2\big]$,
\item[(iii)] $0\leq v_n \leq 1$,
 $|\nabla v_n(x)|\leq cx_d^{-1}$ and
 $|\nabla^2 v_n(x)|\leq cx_d^{-2}$ for $x\in D$.
\end{itemize}
If $\al< 1$, then instead we  
take
$v_n$ 
satisfying
\begin{itemize}
\item[(i')] $v_n=1$ on $[-n^2,n^2]^{d-1}\times [1, n]$,
\item[(ii')] $\supp v_n \subset [-n^2-n,n^2+n]^{d-1}\times
  \big[\frac{1}{2}, 2n\big]$,
\item[(iii)] $0\leq v_n \leq 1$,
 $|\nabla v_n(x)|\leq cx_d^{-1}$ and
 $|\nabla^2 v_n(x)|\leq cx_d^{-2}$ for $x\in D$.
\end{itemize}
Now, for any $\al\in (0,2)$, we define
\begin{equation}
  \label{eq:duu}
u_n(x)=v_n(x)^{\frac{2}{p}}x_d^{\be}\,.
\end{equation}
By Lemma \ref{ep},
\begin{align*}
\EEE_{p}[u_n]=\kappa_{d,p, \al} \int_D \frac{|u_n(x)|^p}{x_d^{\al}} \,dx
+ \frac{1}{p} \int_D\int_D \frac{F_p\left(v_n(x)^{\frac{2}{p} },v_n(y)^{\frac{2}{p}}\right)} {|x-y|^{d+\al}}\,w(x)^{p-1}w(y)\,dx\,dy.
\end{align*}
We have, for $\al \geq 1$,
\begin{equation}\label{lewastr}
\int_D \frac{|u_n(x)|^p}{x_d^\al}\,dx \geq
\int\limits _{\{x:\,\|x'\|\leq n^2,\,\frac{1}{n}<x_d<1\}}
 \frac{x_d^{\al-1}}{x_d^\al}\,dx = (2n^2)^{d-1}\log n,
\end{equation}
and, for $\al<1$,
\begin{equation}
\int_D \frac{|u_n(x)|^p}{x_d^{\al}} dx \geq \int\limits _{\{x:\,\|x'\|\leq n^2,\,1<x_d<n\}}
 \frac{x_d^{\al-1}}{x_d^\al}\,dx = (2n^2)^{d-1}\log n.
\end{equation}
Now, it suffices to show that there exists a constant $c$ 
independent of $n$
such that
$$
\int\limits_D\int\limits_D \frac{F_p\left(v_n(x)^{\frac{2}{p} },v_n(y)^{\frac{2}{p}}\right)} {|x-y|^{d+\al}}\,w(x)^{p-1}w(y)\,dx\,dy \leq
 cn^{2(d-1)}.
$$
To this end, we adapt
\cite[ Proof of Lemma 2.3]{MR2663757}. 
Recall that by \cite[Lemma 2.3]{2020arXiv200601932B}, for $p>1$ and $a,b \in \R$, there exist $c_p, C_p>0$, such that
\begin{equation}\label{Fp-est}
c_p 
\left(b^{\langle p/2 \rangle}-a^{\langle p/2 \rangle}\right)^2
\leq F_p(a,b) \leq 
C_p
\left(b^{\langle p/2 \rangle}-a^{\langle p/2 \rangle}\right)^2
\end{equation}
and hence
\begin{align*}
 \int_D\int_D \frac{F_p\left(v_n(x)^{\frac{2}{p}},v_n(y)^{\frac{2}{p}}\right)} {|x-y|^{d+\al}}\,w(x)^{p-1}w(y)\,dx\,dy &\leq
 C_p
\int_D\int_D \frac{(v_n(x)- v_n(y))^2}{|x-y|^{d+\al}} w(x)^{p-1}w(y)\,dx\,dy\\
&=: C_p I.
\end{align*}
Let $B(x,s,t)=B(x,t)\setminus B(x,s)$. We can bound the latter integral by $cn^{2(d-1)}$, as in \cite{MR2663757}. Considering first $\al\geq 1$, we obtain
\begin{align*} 
I&\leq\int_D \int_{B(x,\frac{1}{4n})} 
+ \int_{\{x:\,x_d\geq \frac{1}{2}\}} \int_{B(x,\frac{1}{4})}
+ \int_D \int_{D\setminus B(x,\frac{1}{4})}
+ \int_{P_n} \int_{D\cap B(x,\frac{1}{4n}, \frac{1}{4})}\\
&\quad+ \int_{R_n} \int_{D \cap B(x,\frac{1}{4n}, \frac{1}{4})}
+\int_{L_n} \int_{D\cap B(x, \frac{1}{4n}, \frac{1}{4})}\\
&=I_1+I_2+I_3+I_4+I_5+I_6,
\end{align*}
where, recalling that $x=(x',x_d)$,
\begin{eqnarray*}
P_n &=& \{x\in\R^d : \,\|x'\|\geq n^2-1\,,\;
0<x_d<\tfrac{1}{2} \} \,, \\ 
R_n &=& \{x\in\R^d : \,\|x'\| < n^2-1\,,\;0<x_d<\tfrac{2}{n} \} \,, \\
L_n&=&\{x\in \R^d : \|x'\| < n^2-1\,,\;\tfrac{2}{n} \leq x_d <\tfrac{1}{2} \} \, ,
\end{eqnarray*}
for $d\geq 2$, and $P_n=\emptyset$, $R_n = \{x\in\R : \,0<x<\frac{2}{n} \}$, $L_n=(\frac{2}{n},\frac{1}{2})$ for $d=1$. For simplicity of notation, let $K_n = \supp v_n$.
Now, estimates for the integrals $I_k$ are analogous to those in
\cite{MR2663757}, although 
we have the non-symmetric factor $w(x)^{p-1}w(y)$
here. For example, if $x\in K_n$ and $y\in B(x,\frac{1}{4n})$, then $|v(x)-v(y)|\leq c|x-y|x_d^{-1}$, as follows from (ii) and (iii). Hence
\begin{eqnarray*}
I_1&=& \int_D \int_{B(x,\frac{1}{4n})}
\frac{(v_n(x)-v_n(y))^2}{|x-y|^{d+\al}}\,w(x)^{p-1}w(y)\,dy\,dx\\
 &\leq&
\int_{K_n} \int_{B(x,\frac{1}{4n})}
\frac{(v_n(x)-v_n(y))^2}{|x-y|^{d+\al}}\,w(x)^{p-1}w(y)\,dy\,dx\\
&\quad& + \int_{K_n} \int_{B(x,\frac{1}{4n})}
\frac{(v_n(x)-v_n(y))^2}{|x-y|^{d+\al}}\,w(x)w(y)^{p-1}\,dy\,dx\\
 &\leq&
c \int_{K_n} \int_{B(x,\frac{1}{4n})} 
  \frac{x_d^{(p-1)\be-2} y_d^{\be}}{|x-y|^{d+\al-2}} \,dy\,dx\\
  &\quad& + c\int_{K_n} \int_{B(x,\frac{1}{4n})} 
  \frac{x_d^{\be-2} y_d^{(p-1) \be}}{|x-y|^{d+\al-2}} \,dy\,dx \\
	&\leq&
c' \int_{K_n} \int_{B(x,\frac{1}{4n})} 
  \frac{x_d^{\al-3} }{|x-y|^{d+\al-2}} \,dy\,dx \\
&\leq& c'' n^{2(d-1)},
\end{eqnarray*}
where in the last line we used the inequality $y_d \leq \frac{3}{2} x_d$, which follows from (ii) and the triangle inequality.
We now turn to the case $\al<1$. We have
\begin{align*}
I&=\int\limits_D\int\limits_D \frac{(v_n(x)-v_n(y))^2} {|x-y|^{d+\al}}\,w(x)^{p-1}w(y)\,dx\,dy\\
&\leq \int_D \int_{B(x,\frac{1}{4})}
+ \int_{\{x:\,x_d\geq \frac{n}{2}\}} \int_{B(x,\frac{n}{4})}
+ \int_D \int_{D\setminus B(x,\frac{n}{4})}
+ \int_{P_n} \int_{D\cap B(x,\frac{1}{4}, \frac{n}{4})}\\
&\quad+ \int_{\{x:0<x_d<2\}} \int_{D \cap B(x,\frac{1}{4}, \frac{n}{4})}
+ \int_{L_n} \int_{D\cap B(x, \frac{1}{4}, \frac{n}{4})}\\
&=I_1+I_2+I_3+I_4+I_5+I_6,
\end{align*}
where
\begin{eqnarray*}
P_n &=& \{x\in\R^d : \,\|x'\|\geq n^2-n\,,\; 0<x_d<\tfrac{n}{2} \} \, ,\\
L_n&=&\{x\in \R^d : \|x'\| < n^2-n\,,\; 2 \leq x_d <\tfrac{n}{2} \} \, ,
\end{eqnarray*}
for $d\geq 2$, and $P_n=\emptyset$, $L_n=(2,\frac{n}{2})$ for $d=1$.

The integrals $I_k$ can be estimated similarly 
as in the case $\al \geq 1$.

\end{proof}

\section{Proof of Theorem \ref{convex}}

Frank and Seiringer \cite{MR2469027} proved an abstract form of the fractional Hardy inequality in a~more general setting. Loosely speaking, they showed that under some assumptions the inequality \begin{equation}\label{genhardyfrankser}
\int_{\Omega}\int_{\Omega}|u(x)-u(y)|^{p}\,k(x,y)\,dy\,dx\geq\int_{\Omega}|u(x)|^{p}\,V(x)\,dx
\end{equation}
holds for symmetric kernels $k(x,y)$ and related 
functions $V$ (see \cite[Proposition 2.2]{MR2469027}). 
In order to prove Theorem \ref{convex}, we first need to formulate an analogue of \eqref{genhardyfrankser} for Sobolev-Bregman forms.

Let $\Omega\subset\R^d$ be a nonempty, open set. Suppose that $\{k_{\varepsilon}(x,y)\}_{\varepsilon>0}$ is a family of measurable, symmetric kernels satisfying
$$
0\leq k_{\varepsilon}(x,y)\leq k(x,y),\quad \lim_{\varepsilon\rightarrow 0^+}k_\varepsilon (x,y)=k(x,y),
$$
 for almost all $x,y$ and some measurable function $k(x,y)$. Moreover, let $w$ be a positive, measurable function. Define
\begin{equation}\label{Veps}
 V_{\varepsilon}(x):=\int_{\Omega}\left(\frac{1}{p}\frac{w(x)^{p-1}-w(y)^{p-1}}{w(x)^{p-1}}+\frac{p-1}{p}\frac{w(x)-w(y)}{w(x)}\right)k_{\varepsilon}(x,y)\,dy 
\end{equation}
and suppose that there exists a function $V$ such that $V_{\varepsilon}\rightarrow V$ weakly in $L^{1}_{loc}(\Omega)$, that is, for any bounded $g$ with compact support in $\Omega$, $\int_{\Omega}V_{\varepsilon}(x)g(x)\,dx\rightarrow\int_{\Omega}V(x)g(x)\,dx$  as $\varepsilon\rightarrow 0$.

In what follows, we will use the notation
$$\EEE_{p}[u]=\EEE_{p}^{\Omega,k}[u]:=\frac{1}{2}\int_{\Omega}\int_{\Omega}(u(x)-u(y))\left(u(x)^{\langle p-1 \rangle}-u(y)^{\langle p-1 \rangle}\right)k(x,y)\,dy\,dx.$$
\begin{lemma}\label{genhardy}
For $u\in C_{c}(\Omega)$,
\begin{equation}\label{generalhardy}
 \EEE_{p}[u]\geq\int_{\Omega}|u(x)|^{p}\,V(x)\,dx.
\end{equation}
\end{lemma}
\begin{proof}
Proceeding as in the proof of Lemma \ref{ep}, we arrive at the equality
\begin{align*}
\frac{1}{2}\int_{\Omega}\int_{\Omega}&
(u(x)-u(y))\left(u(x)^{\langle p-1 \rangle}-u(y)^{\langle p-1 \rangle}\right)
k_{\varepsilon}(x,y)\,dy\,dx&\\
&=\int_{\Omega}|u(x)|^{p}\,V_{\varepsilon}(x)\,dx+\frac{1}{p}\int_{\Omega}\int_{\Omega}F_{p}\left(\frac{u(x)}{w(x)},\frac{u(y)}{w(y)}\right)w(x)^{p-1}w(y)k_{\varepsilon}(x,y)\,dy\,dx\\
&\geq \int_{\Omega}|u(x)|^{p}\,V_{\varepsilon}(x)\,dx,
\end{align*}
since $F_{p}(a,b)\geq 0$. Now we let $\varepsilon\rightarrow 0$ and use Lebesgue's Dominated Convergence Theorem on the left-hand side (provided that $\EEE_{p}[u]<\infty$) and weak convergence on the right-hand side to obtain the desired result.
\end{proof}
\begin{lemma}\label{lem3}
Let $1<\al<2$ and $J\subset\R$ be an open set. Then for $u\in C_{c}(J)$,
\begin{equation}\label{J}
\frac{1}{2}\int_{J}\int_{J}\frac{(u(x)-u(y))\left(u(x)^{\langle p-1 \rangle}-u(y)^{\langle p-1 \rangle}\right)}{|x-y|^{1+\al}}\,dy\,dx\geq\kappa_{1,p,\al}\int_{J}\frac{|u(x)|^{p}}{\dist(x,\partial J)^{\al}}\,dx.
\end{equation}
\end{lemma}
\begin{proof}
Our proof relies on an appropriate modification of \cite[Proof of Theorem 2.5]{MR2659764}. 
We will first consider the case $J=(0,1)$.
Set 
$w=w_{(\al-1)/p}$ and, for $x\in(0,1)$, 
$$
V(x)=\text{P.V.}\int_{0}^{1}\left(\frac{1}{p}\frac{w(x)^{p-1}-w(y)^{p-1}}{w(x)^{p-1}}+\frac{p-1}{p}\frac{w(x)-w(y)}{w(x)}\right)|x-y|^{-1-\al}\,dy.
$$
By \eqref{ullapl}, 
we have 
\begin{align*}
V(x)&=\text{P.V.}\left(\int_{0}^{\infty}-\int_{1}^{\infty}\right)\left(\frac{1}{p}\frac{w(x)^{p-1}-w(y)^{p-1}}{w(x)^{p-1}}+\frac{p-1}{p}\frac{w(x)-w(y)}{w(x)}\right)|x-y|^{-1-\al}\,dy\\
&\geq \text{P.V.}\int_{0}^{\infty}\left(\frac{1}{p}\frac{w(x)^{p-1}-w(y)^{p-1}}{w(x)^{p-1}}+\frac{p-1}{p}\frac{w(x)-w(y)}{w(x)}\right)|x-y|^{-1-\al}\,dy\\
&=\kappa_{1,p,\al}x^{-\al},
\end{align*}
since the integrand is nonpositive for $y\in[1,\infty)$. In addition, the latter principal value integral is uniformly convergent on every compact set $K \subset (0,\infty)$. In consequence, 
$$V_{\varepsilon}(x)=\left(\int_{0}^{x-\varepsilon}+\int_{x+\varepsilon}^{1}\right)\left(\frac{1}{p}\frac{w(x)^{p-1}-w(y)^{p-1}}{w(x)^{p-1}}+\frac{p-1}{p}\frac{w(x)-w(y)}{w(x)}\right)|x-y|^{-1-\al}\,dy>0$$ 
for small $\varepsilon>0$ and all $x\in K$. Hence, by the proof of Lemma \ref{genhardy} combined with Fatou's lemma,
\begin{equation}\label{01}
\frac{1}{2}\int_{0}^{1}\int_{0}^{1}\frac{(v(x)-v(y))\left(v(x)^{\langle p-1 \rangle}-v(y)^{\langle p-1 \rangle}\right)}{|x-y|^{1+\al}}\,dy\,dx\geq\kappa_{1,p,\al}\int_{0}^{1}\frac{|v(x)|^{p}}{x^{\al}}\,dx,
\end{equation}
for any function $v$ such that $\supp v\subset (0,1]$. Moreover, for $u\in C_c((0,1))$, observe that using \eqref{01} gives
\begin{align*}
&\int_{0}^{1}\frac{|u(x)|^{p}}{\min\{x,1-x\}^{\al}}\,dx  \\
&\quad\quad=\int_{0}^{\frac{1}{2}}\frac{|u(x)|^{p}}{x^{\al}}\,dx+\int_{\frac{1}{2}}^{1}\frac{|u(x)|^{p}}{(1-x)^{\al}}\,dx\\
&\quad\quad=2^{\al-1}\left(\int_{0}^{1}\frac{\big|u\big(\frac{x}{2}\big)\big|^{p}}{x^{\al}}\,dx+\int_{0}^{1}\frac{\big|u\big(1-\frac{x}{2}\big)\big|^{p}}{x^{\al}}\,dx\right)\\
&\quad\quad\leq 2^{\al-1}\kappa_{1,p,\al}^{-1}
\Bigg(\int_{0}^{1}\int_{0}^{1}\frac{\Big(u\big(\frac{x}{2}\big)-u\big(\frac{y}{2}\big)\Big)\Big(u\big(\frac{x}{2}\big)^{\langle p-1\rangle}-u\big(\frac{y}{2}\big)^{\langle p-1\rangle}\Big)}{|x-y|^{1+\al}}\,dy\,dx\\
&\quad\quad\quad+\int_{0}^{1}\int_{0}^{1}\frac{\Big(u\big(1-\frac{x}{2}\big)-u\big(1-\frac{y}{2}\big)\Big)\Big(u\big(1-\frac{x}{2}\big)^{\langle p-1\rangle}-u\big(1-\frac{y}{2}\big)^{\langle p-1\rangle}\Big)}{|x-y|^{1+\al}}\,dy\,dx\Bigg)\\
&\quad\quad=\kappa_{1,p,\al}^{-1}\left(\int_{0}^{\frac{1}{2}}\int_{0}^{\frac{1}{2}}+\int_{\frac{1}{2}}^{1}\int_{\frac{1}{2}}^{1}\right)\frac{\left(u(x)-u(y)\right)\big(u(x)^{\langle p-1\rangle}-u(y)^{\langle p-1\rangle}\big)}{|x-y|^{1+\al}}\,dy\,dx\\
&\quad\quad\leq\kappa_{1,p,\al}^{-1}\int_{0}^{1}\int_{0}^{1}\frac{\left(u(x)-u(y)\right)\left(u(x)^{\langle p-1\rangle}-u(y)^{\langle p-1\rangle}\right)}{|x-y|^{1+\al}}\,dy\,dx.
\end{align*}

By translation and scaling, an analogous formula holds for any interval $(a,b)$ and since every open subset of $\R$ is a countable union of disjoint intervals, \eqref{J} is an easy consequence of the above computations.
\end{proof}

\begin{proof}[Proof of Theorem \ref{convex}]
We will use arguments similar to those presented by Loss and Sloane in \cite{MR2659764}. We denote by ${\mathcal L}_\omega$ the $(d-1)$ dimensional Lebesgue measure on the plane $x \cdot \omega =0$. Calculations analogous to \cite[Proof of Lemma 2.4]{MR2659764} and Lemma \ref{lem3} give

\begin{align*}
\frac{1}{2}&\int_{\Omega}\int_{\Omega} \frac{ (u(x)-u(y))(u(x)^{\langle p-1 \rangle} - u(y)^{\langle p-1 \rangle})}{|x-y|^{d+\al}} \, dx \, dy \\
&= \frac{1}{4}\int_{\Ss^{d-1}}\,d\omega \int_{\{x: x \cdot \omega = 0\}}\,d{\mathcal L}_\omega(x)  \int_{\{x+s\omega \in \Omega\}}\,ds  \int_{\{x+t\omega \in \Omega\}}
\\
&\quad
\times
\frac{ (u(x+s\omega)-u(x+t\omega))\left(u(x+s\omega)^{\langle p-1 \rangle} - u(x+t\omega)^{\langle p-1 \rangle}\right)}{{|s-t|}^{1+\al}}\,dt\\
&\geq \frac{1}{2} \kappa_{1,p,\al}\int_{\Ss^{d-1}}\int_{\{x: x \cdot \omega = 0\}}\int_{\{x+s\omega\in \Omega\}}\frac{|u(x+s\omega)|^{p}}{d_{\omega,\Omega}(x+s\omega)^{\al}}\,ds\,d{\mathcal L}_\omega(x)\,d\omega\\
&=  \frac{1}{2} \kappa_{1,p,\al}\int_{\Ss^{d-1}}\int_{\Omega}\frac{|u(x)|^{p}}{d_{\omega,\Omega}(x)^{\al}}\,dx\,d\omega\\
&=\kappa_{d,p,\al}\int_{\Omega}\frac{|u(x)|^{p}}{m_{\alpha}(x)^{\al}}\,dx,
\end{align*}
where the last equality follows from
$$
\int_{\Ss^{d-1}}|\omega_{d}|^{\al}\,d\omega =\frac{2 \pi^{\frac{d-1}{2}}\Gamma\big(\tfrac{1+\al}{2}\big)}{\Gamma\big(\tfrac{d+\al}{2}\big)}.
$$
This proves \eqref{hardyconvex1} and, in consequence, \eqref{hardyconvex2}. 
Since $\Omega$ is convex, there exists a hyperplane $\Pi$ tangent to $\Omega$ at a point $P \in \partial \Omega$. Now, calculations analogous to those in \cite[Proof of Theorem 5]{MR1458330} yield the optimality of the constant in \eqref{hardyconvex2}.
\end{proof}
\begin{remark}\label{remark}
Noteworthy, if $\al \leq 1$ and $\Omega$ is a bounded convex domain, then the best constant in the inequality \eqref{hardyconvex1} is zero. Indeed, first notice that every convex set is a Lipschitz domain. Dyda constructed in \cite{MR2085428} a sequence of functions $u_{n}\in C_{c}(\Omega)$ such that $0\leq u_n\leq 1$, $u_{n}\rightarrow 1$ pointwise and $\int_{\Omega}\int_{\Omega}\frac{|u_{n}(x)-u_{n}(y)|^{2}}{|x-y|^{d+\al}}\,dy\,dx\rightarrow 0$ for $\al<1$, as $n\rightarrow\infty$, $\int_{\Omega}\int_{\Omega}\frac{|u_{n}(x)-u_{n}(y)|^{2}}{|x-y|^{d+\al}}\,dy\,dx\leq C$ for $\al=1$. Taking $v_{n}=u_{n}^{2/p}$ and using \eqref{Fp-est}, we see that the inequality \eqref{hardyconvex1} cannot hold with a positive constant, as the right-hand side of \eqref{hardyconvex1} tends to a positive value, when $\al<1$ and to infinity, when $\al=1$.
\end{remark}

\bibliographystyle{acm}

\end{document}